\documentclass[10pt,reqno]{amsart}% use larger type; default would be 10pt

\usepackage[utf8]{inputenc} % set input encoding (not needed with XeLaTeX)
\usepackage[margin=1in]{geometry} % to change the page dimensions
\usepackage{graphicx} % support the \includegraphics command and options
\usepackage{float} %have graphics not floating around

\usepackage{changepage}

 \usepackage[parfill]{parskip} % Activate to begin paragraphs with an empty line rather than an indent
 
\usepackage{booktabs} % for much better looking tables
\usepackage{array} % for better arrays (eg matrices) in maths
\usepackage{paralist} % very flexible & customisable lists (eg. enumerate/itemize, etc.)
\usepackage{verbatim} % adds environment for commenting out blocks of text & for better verbatim
\usepackage{subfig} % make it possible to include more than one captioned figure/table in a single float
\usepackage{mathrsfs}
\usepackage{amssymb}
\usepackage{xcolor}
\usepackage{amsthm}
\usepackage{amsmath,amsfonts,amssymb,esint,hyperref}
\usepackage[noabbrev, capitalize]{cleveref}

\usepackage{autonum}

\usepackage{graphics,color}
\usepackage{enumerate, enumitem}
\usepackage{mathtools,centernot}
\usepackage{cases}
\usepackage{amsrefs}
\usepackage{bbm}
\usepackage{xfrac}

\usepackage{bookmark}

\newtheorem{theorem}{Theorem}[section]
\newtheorem{lemma}[theorem]{Lemma}

\newtheorem{definition}[theorem]{Definition}
\newtheorem{proposition}[theorem]{Proposition}
\newtheorem{remark}[theorem]{Remark}

\numberwithin{equation}{section} 

\newcommand{\norm}[1]{\left\|#1\right\|}
\newcommand{\abs}[1]{\left|#1\right|}

\newcommand*{\R}{\ensuremath{\mathbb{R}}}

\renewcommand*{\S}{\ensuremath{\mathcal{S}}}

\newcommand*{\tr}{\ensuremath{\mathrm{tr\,}}}

\renewcommand{\phi}{\varphi}
\newcommand{\J}{\mathcal{J}}

\renewcommand{\MR}[1]{} %remove MR from biblio

\usepackage{color, graphicx}
\usepackage{mathrsfs, dsfont}

\usepackage[]{hyperref}
\hypersetup{
    colorlinks=true,       % false: boxed links; true: colored links
    linkcolor=red,          % color of internal links
    citecolor=blue,        % color of links to bibliography
    filecolor=red,      % color of file links
    urlcolor=cyan           % color of external links
}

\def\div{\mathop{\rm div}\nolimits}    %divergence
\def\curl{\mathop{\rm curl}\nolimits}    %curl
 
\def\tr{\mathop{\rm Tr}\nolimits}

\newcommand{\be}{\begin{equation}}
\newcommand{\ee}{\end{equation}}
\newcommand{\D}{\mathrm{D}}

\title{Fine dissipative properties of Euler solutions with measure first derivatives}

\author[M. Inversi]{Marco Inversi}
\address[M. Inversi]{Max Planck Institute for Mathematics in the Sciences, Inselstrasse 22, 04103, Leipzig, Germany.}
\email{marco.inversi@mis.mpg.de}

\date{\today}

\begin{document}

\begin{abstract}
We study fine properties of bounded weak solutions to the incompressible Euler equations whose first derivatives, or only some combinations of them, are Radon measures. As consequences we obtain elementary proofs of the local energy conservation for solutions in BV and BD, without relying on the freedom in choosing the convolution kernel appearing in the approximation of the dissipation. The argument heavily exploits the form of the Euler nonlinearity and it does not apply to the linear transport equations, where the renormalization property for BD vector fields is an open problem. 
The methods also yield nontrivial conclusions when only the vorticity is assumed to be a measure. 
\end{abstract}

\maketitle

\section{Introduction} 

The motion of a homogeneous, incompressible, ideal fluid under the action of an external force field $f$ is governed by the incompressible Euler equations
\begin{equation} \label{E} \tag{E}
    \begin{cases}
    \partial_t u + \div (u \otimes u) + \nabla p =f  
    \\ \div (u) = 0
    \end{cases} \qquad \text{in } \Omega \times (0,T).
\end{equation}
The system is settled in $\Omega \times (0,T)$, where $\Omega \subset \R^d$ is an open set. Since we focus on the analysis in the interior of the domain, the spatial boundary conditions and the initial datum are not specified. Here, $u$ is the velocity field of the fluid and $p$ is the hydrodynamic pressure. From now on, we will tacitly assume that an open set $\Omega \subset \R^d$ and a time interval are fixed and $u \in L^\infty(\Omega \times (0,T))$ (in short $u \in L^\infty_{x,t}$), $ f, p \in L^{1}_{x,t}$. Weak solutions to \eqref{E} satisfy the local energy balance
\begin{equation} \label{eq: LEB}
    \partial_t \frac{\abs{u}^2}{2} + \div \left( u \left( \frac{\abs{u}^2}{2} + p \right) \right) = -\D + f \cdot u \qquad \text{in } \mathcal{D}'(\Omega \times (0,T)), 
\end{equation}
up to a possibly non-trivial defect distribution $\D$, named after Duchon and Robert \cite{DR00}. Smooth solutions to \eqref{E} satisfy the exact energy balance, i.e. \eqref{eq: LEB} holds with $\D \equiv 0$.  Since the celebrated works of Kolmogorov \cite{K41} and
Onsager \cite{O49}, it has been clear that a good understanding of incompressible
turbulence is subject to the study of Euler solutions violating energy conservation. The results established in this note are related to the celebrated Onsager conjecture in critical regimes and the analysis of vortex sheets type solutions. More precisely, we study the dissipation associated with bounded solutions to \eqref{E} whose first derivatives, or only some combinations of them, are represented by Radon measures. See \cref{s: tools} for precise definitions. All the results and the computations described in this note are specific for the incompressible Euler system. To begin, we give a direct proof of the following result, recently obtained by the author and De Rosa in \cite{DRINV23}. 

\begin{theorem} \label{t:main intro 1}
Let $u \in L^\infty_{x,t} \cap L^1_t BV_x $ be a weak solution to \eqref{E} with $p,f \in L^1_{x,t}$. Then $\D \equiv 0$. 
\end{theorem}

The above theorem has been proved in \cite{DRINV23} by a careful analysis of the Duchon--Robert approximation, following the optimization procedure on the convolution kernel proposed by Ambrosio \cite{Ambr04} to establish the renormalization property for the transport equation with $BV$ vector fields. We propose a proof based only on the chain rule for divergence-free $BV$ vector fields. The computations are specific for the Euler nonlinearity. We postpone further discussions to \cref{ss: bounded BV} and \cref{s: direct proof}. 

Then, we consider the case of bounded deformation velocity fields and will prove the following. 

\begin{theorem} \label{thm: main BD}
Let $u \in L^\infty_{x,t} \cap L^1_t BD_x$ be a weak solution to \eqref{E} with $p, f \in L^1_{x,t}$. Then $\D \equiv 0$. 
\end{theorem}

The above result clearly extends \cref{t:main intro 1}, but the proof is not a technical modification of the previous. In the bounded deformation setting, the $BV$ chain rule and the Ambrosio's optimization procedure on the convolution kernels are not available. To the best of the author's knowledge, it is not clear how to borrow ideas from the linear setting, since the   renormalization property for the transport equation with $BD$ vector fields is currently open in full generality. See e.g. the discussion in \cite{Cr09}. Here, the argument exploits the fine properties of the blow up of divergence-free $BD$ vector fields, combined with the specific form of Euler nonlinearity. \cref{thm: main BD} also concludes the analysis started in \cite{DRInvN24}, since no further assumptions on the singular part of the symmetric gradient are needed. See \cref{ss: BD velocity} and \cref{s: velocity BD} for more details.  
 
In the last part of this note, we address the case of bounded Euler solutions with measure vorticity (or antisymmetric gradient). The analysis of the energy conservation in this class was initiated by Shvydkoy in \cite{Shv09} for classical vortex sheet solutions. In a nutshell, we will show that, whenever a (suitably defined) singular set is negligible with respect to the measure vorticity, then \eqref{eq: LEB} holds with $\D \equiv 0$. See \cref{t:main general} for a precise statement. Even in this case, it does not seem possible to draw inspiration from the linear theory, since the renormalization of the transport equation with measure vorticity is currently open in full generality. Again, the result and the computations heavily relies on the form of the Euler nonlinearity and the divergence-free condition. Further discussions are postponed to \cref{ss: generalized vortex sheets} and \cref{s: vortex sheets}.

\subsection{Plan of the paper}

The remainder of this paper is structured as follows. In \cref{s: background} we describe the relation of our results with the Onsager conjecture and the previous literature, as well as the main ideas behind our proofs. In \cref{s: tools} we recall some useful notation and basic tools. \cref{s: proofs} is devoted to proofs. 

\subsection{Acknowledgments} The author has been partially supported by the SNF grant FLUTURA: Fluids, Turbulence, Advection No. 212573. The author thanks Luigi De Rosa for a careful reading of a preliminary version of the current paper and for his constant interest in discussing/sharing mathematics. 

\section{Background and heuristics} \label{s: background}

\subsection{The Onsager conjecture} In 1949, Onsager proposed $\sfrac{1}{3}$ of H\"older regularity on the velocity field as a threshold for the conservation of kinetic energy \cite{O49}. This is known in the literature as the \emph{Onsager conjecture}. In the last decades, mathematical research has confirmed Onsager's prediction. Within the framework of Besov spaces, it is well-known that $\D \equiv 0$ whenever $u \in L^3 B^\theta_{3,\infty}$ for $\theta > \sfrac{1}{3}$ \cites{E94, CET94, DR00}. These classical proofs are based on the regularization of the equation by convolution. Due to the quadratic nature of the advective term in \eqref{E}, after taking the scalar product by the (possibly mollified) vector field, a trilinear error term appears. More precisely, in \cite{CET94} it is proved that $\D$ is the distributional limit of the sequence\footnote{The subscript index $\ell$ denotes the convolution with a mollifier $\rho_\ell = \ell^{-d} \rho\left( \sfrac{\cdot}{\ell}\right)$, where $\rho \in \mathcal{K}$ \eqref{eq: kernels} Here $\rho$ is smooth, even, with unit mass and supported in $B_1$. } 
\begin{equation} \label{eq: CET approx}
    \D_\ell^{CET}(x,t) := R^\ell(x,t) : \nabla u_\ell(x,t), \qquad \text{where} \quad R^\ell : = u_\ell \otimes u_\ell - (u \otimes u)_\ell. 
\end{equation}
By commutator estimates, then $\D^{CET}_\ell \to 0$ strongly in $L^1_{x,t}$ as soon as $u \in L^3 B^{\theta}_{3,\infty}$, for some $\theta > \sfrac{1}{3}$. In \cite{DR00}, Duchon and Robert obtained the same result writing $\D$ as the distributional limit of the sequence 
\begin{equation} \label{eq: DR approx}
    \D_\ell^{DR}(x,t) := \frac{1}{4 \ell} \int_{B_1} \nabla \rho(z) \cdot \delta_{\ell z} u(x,t) \abs{\delta_{\ell z} u(x,t)}^2 \, dz,  \qquad \text{where}\quad \delta_{\ell z} u(x,t) := u(x+\ell z, t)- u(x,t). 
\end{equation}
On the opposite side, several constructions of non-conservative (often truly dissipative) solutions are currently available. After Scheffer \cite{Sch93} and Shnirelman \cite{Shn00}, De Lellis and Sz\'ekelyhidi \cites{DS09, DS13} introduced convex integration methods in fluid dynamics to produce pathological solutions to \eqref{E}. Finally, the negative part of the Onsager conjecture has been settled by Isett \cite{Is18}, who constructed Euler solutions in $C^\theta$ with $\D \not\equiv 0$, for any $\theta < \sfrac{1}{3}$. See also \cites{NV22, GKN24_1, GKN24_2} for further developments. In summary, the Onsager conjecture is fully proved in both the supercritical and the subcritical regimes. 

To the best of the author's knowledge, the analysis in critical spaces such as $C^{\sfrac{1}{3}}$, as well as $B^{\sfrac{1}{3}}_{p, \infty}$ for any $p \in [3, +\infty)$,  is far from complete. Here we refer to \emph{Onsager critical class} as any function space which makes the approximate energy flux bounded\footnote{This is defined, for instance, by the Constantin--E--Titi approximation \eqref{eq: CET approx} or the Duchon--Robert one \eqref{eq: DR approx}.}, but possibly non vanishing. In these cases, the dissipation $\D$ is a space-time Radon measure, possibly non-trivial. On the other hand, convex integration encounters serious technical obstructions in  $B^{\sfrac{1}{3}}_{p, \infty}$ for $p \geq 3$ and, currently, the method fails. In \cites{CCFS08, DRInvN24} it was observed that in Onsager critical classes, $\D \equiv 0$ holds under the additional assumption that \emph{oscillations at small scales are small, when measured in a critical norm}\footnote{This is equivalent to consider the closure of smooth functions in critical classes, which is typically a proper subspace of the ambient space.}. See also \cite{DRDrIsInv25} for further properties of the dissipation $\D$ and its relation with intermittency in fluid dynamics.

\subsection{Bounded BV solutions} \label{ss: bounded BV}

To the author's knowledge, $L^\infty_{x,t} \cap L^1_t BV_x$ is the only critical class for which a complete characterization of the dissipation is currently available. By interpolation $BV \cap L^\infty$ embeds sharply in $B^{\sfrac{1}{3}}_{3,\infty}$\footnote{In fact, $BV \cap L^\infty$ embeds sharply in $B^{\sfrac{1}{p}}_{p, \infty}$ for any $p > 1$. Thus, the solutions in this class exhibit intermittency. }, which makes all the classical results \cites{CET94, DR00, CCFS08} unavailable in this setting. As already mentioned, in \cite{DRINV23} a first proof of \cref{t:main intro 1} appeared, based on the following heuristic argument. The Duchon--Robert approximation \eqref{eq: DR approx} makes the role of the convolution kernel $\rho$ apparent. Nevertheless, the limit distribution $\D$ is independent of the choice of $\rho$, which can be optimized. Following the seminal paper by Ambrosio \cite{Ambr04} on the renormalization of the transport equation with $BV$ vector fields, it is shown in \cite{DRINV23} that the defect measure $\D$ associated to an Euler solution in $L^\infty_{x,t} \cap L^1_t BV_x$ satisfies 
\begin{equation}
    \abs{\D} \lesssim \left( \inf_{\rho} \int_{B_1} \abs{\nabla \rho(z) \cdot ( M_{x,t \,} z) } \, dz \right) \lambda.  
\end{equation} 
Here $\lambda$ is a space-time Radon measure and, as a direct consequence of the incompressibility condition, $ M_{x,t} $ is a trace-free matrix $\lambda$-almost everywhere\footnote{The proof provides $\lambda = \abs{D u_t }\otimes dt$ and $M_{x,t}$ is the Radon-Nykodim derivative of $D u_t \otimes dt$ with respect to $\lambda$.}. The infimum is computed among all admissible convolution kernels and a careful computation by Alberti\footnote{In the original paper by Ambrosio \cite{Ambr04} the Rank One theorem for $BV$ functions is exploited. Later, the proof has been simplified \cite{Ambr04}*{Remark 3.7}. } \cites{Cr09, DRINV23} shows that it equals $\abs{\tr(M_{x,t})}$, thus proving $\D \equiv 0$.   

Some remarks are in order. Unlike the classical proofs, this argument shows \emph{weak} convergence of the approximate energy flux to $0$, opposite to \emph{strong} $L^1$. Moreover, the analysis in the critical space considered here makes a fine use of the divergence-free condition. It is clear that the proof must fail for compressible models, for instance when considering $1d$ Burgers shocks. Indeed, there exist bounded entropic solutions with $BV$ regularity (in space-time) and non-trivial defect measure. As a matter of fact, Burgers equation has a quadratic nonlinearity, which dictates the same critical classes for the conservation of energy as the incompressible Euler system. 

The proof outlined above is deeply inspired by the theory of renormalized solution for the transport equation. However, despite the highly non-trivial use of the divergence-free condition, the argument does not involve specific manipulations on the momentum equation. For instance, the fact that the Duchon--Robert approximation is trilinear with respect to the increment is simply neglected. In \cref{s: direct proof} we give a direct proof of \cref{t:main intro 1} based only on the chain rule for $BV$ functions. By the discussion above, the divergence-free condition must be used and it will force non-trivial cancellations to happen. The key point is that bounded divergence-free $BV$ vector fields satisfy
$$\div \left (u \frac{\abs{u}^2}{2}\right) = [ (u \cdot \nabla ) \, u ] \cdot u \qquad \text{in the sense of measures on } \R^d, $$
as in the smooth setting. In \cref{R: on the failure of the chain rule}, we will comment on the fact that the very same computations for Burgers shocks yield a non-trivial dissipation measure due to the jumps of the velocity.

\subsection{BD velocity fields} \label{ss: BD velocity}

The arguments in \cites{DR00, CET94} can be generalized to the case of bounded solutions with symmetric gradient in $L^1$ or continuous solutions with measure symmetric gradient (see e.g. \cite{DRInvN24}). Moreover, using  a radial kernel in \eqref{eq: DR approx}, in \cite{DRInvN24}*{Corollary 4.2} it is shown that for solutions in $L^\infty_{x,t} \cap L^1_t BD_x$, the dissipation is controlled by the singular part of the symmetric gradient of the velocity field. Then, under  additional assumptions on the geometric structure of the symmetric singular gradient, it is possible to conclude that $\D \equiv 0$ (see \cite{DRInvN24}*{Corollary 1.4}). In \cref{thm: main BD}, the unconditioned $BD$ case is addressed, concluding the analysis started in \cite{DRInvN24}. This also provides an energy conservation result in a critical class where only an assumption on the \emph{longitudinal structure function exponents} is made\footnote{Longitudinal and absolute structure function exponents are used in the theory of intermittency, see e.g. \cite{DRDrIsInv25} for further discussions. In view of \eqref{eq: char of BD}, $BD$ vector fields correspond to the first order \emph{longitudinal} structure function exponent $\zeta_1^{\parallel} = 1$, whereas $BV$ vector fields correspond to the first order \emph{absolute} structure function exponent $\zeta_1 = 1$ by \eqref{eq: char of BV}.}. 

In this case, it is not clear how to adapt known techniques from the linear setting, since the renormalization theory for the transport equation with $BD$ vector fields is open in full generality\footnote{For instance, the Alberti--Ambrosio optimization procedure is not available in this setting. Indeed, in order to exploit the $BD$ regularity, we are forced to use radial kernels, whereas the optimization proposed in \cite{Ambr04} requires one to use anisotropic kernels.}. See e.g. \cites{ACM05, Cr09} for partial results and further discussions. The proof of \cref{thm: main BD} is specific for Euler nonlinearity. A careful analysis of the blow up of $BD$ vector fields and the divergence-free condition are required. Given a radial convolution kernel $\rho$, after some manipulations of the Constantin--E--Titi formula \eqref{eq: CET approx}, we write 
\begin{equation}
    \D = \lim_{\ell \to 0} (R^\ell)_\ell : Eu, \qquad \text{where } R^\ell = u_\ell \otimes u_\ell - (u\otimes u)_\ell, \quad Eu = \frac{Du + (Du)^T}{2}.  
\end{equation}
For $BD$ velocity fields, the pointwise limit of $(R^\ell)_\ell$ can be explicitly computed $\abs{Eu}$-almost everywhere and it does not vanish only on the jump set $\J_u$ of the velocity field. Then, by the divergence-free condition and the characterization of the symmetric gradient, we will show that the dissipation does not give mass to $\J_u$. This suffices to conclude that $\D \equiv 0$.

\subsection{Solutions with measure vorticity} \label{ss: generalized vortex sheets}

In \cref{s: vortex sheets} we will consider bounded Euler solutions whose antisymmetric gradient is an arbitrary Radon measure, such as classical vortex sheets. These are Euler solutions with vorticity concentrated on a hypersurface of codimension $1$ that evolves over time \cite{Saf92}. To the best of the author's knowledge, the first energy conservation result for classical vortex sheets is due to Shvydkoy \cite{Shv09}. The energy flux associated with these solutions is a priori non-vanishing. However, being a solution to \eqref{E} forces the normal component of the velocity and the Bernoulli pressure to be continuous across the sheet. This suffices to prove that there is no energy dissipation in any space-time region. This heuristic has been recently adapted in \cite{DRInvN24} to prove that there is no possibility for energy gaps whenever singularities occur on a codimension $1$ rectifiable set, provided that both the velocity and the pressure have suitable traces.

In \cref{t:main general} we will study the dissipation measure associated with Euler solutions with measure antisymmetric gradient. Even for vector fields with this regularity, the linear theory does not provide many tools. The renormalization property for the transport equation is known for divergence-free vector fields with gradient given by a singular integral of a $L^1$ function \cite{BCr13} or in the stationary two-dimensional setting for measure vorticity vector fields \cite{BiGu16}. In both cases, the arguments are rather deep and they do not seem extendable to the Euler setting. 

Once more, the proof of \cref{t:main general} is specific for the Euler nonlinearity. The key point is to use the incompressibility condition to manipulate the momentum equation so that the role of the antisymmetric gradient is apparent. Under the non-trivial assumption that a (suitably defined) singular set of the velocity field is negligible with respect to the vorticity measure, we will show that  
\begin{equation}
    \partial_t u + \nabla P  = - (\Lambda u) \, u + f \qquad \text{where } P = \frac{\abs{u}^2}{2} + p, \quad \Lambda u = \nabla u - (\nabla u)^T. \label{eq: modified euler}
\end{equation}
Here, the product $(\Lambda u) u$ turns out to be a well-defined measure. Then, we formally take the scalar product by the velocity field $u$ and, again using the divergence-free condition to write $u \cdot \nabla P = \div (uP)$, we obtain
\begin{equation}
    \partial_t \frac{\abs{u}^2}{2} + \div (uP) = -[(\Lambda u)\, u] \cdot u + f\cdot u = -(\Lambda u) : u \otimes u + f\cdot u = f\cdot u 
\end{equation}
by a simple orthogonality argument. To make this argument rigorous, two mollification procedures will be needed. See \cref{s: vortex sheets} for details. This gives another proof of \cref{t:main intro 1} (see \cref{R: particular cases}). 

\section{Tools} \label{s: tools}

We introduce some useful notation that will be kept throughout the manuscript. Given a weak solution $(u,p)$ to \eqref{E} with force $f$, we will often consider the Bernoulli pressure 
$$P : = \frac{\abs{u}^2}{2} + p. $$
Given a vector field $u: \R^d \to \R^d$, we denote by $Du, E u, \Lambda u $ the distributional gradient, its symmetric, and (twice) its antisymmetric part, respectively. More precisely, we have 
$$ (E u)_{i,j} = \frac{[Du + (Du)^T]_{ij}}{2} = \frac{\partial_j u_i + \partial_i u_j}{2} \in \R^{d \times d}, \qquad  (\Lambda u)_{ij} : = [D u - (D u)^T]_{ij} = \partial_j u_i - \partial_i u_j \in \R^{d\times d}.  $$
For two dimensional flows, the antisymmetric part of the gradient is identified with the scalar vorticity
\begin{equation}
    \Lambda u = \left( \begin{matrix}
        0 & -\omega 
        \\ \omega & 0 
        \end{matrix}
        \right), \qquad \text{where } \omega = \partial_1 u_2 - \partial_2 u_1. 
\end{equation}
In particular, for any vector $v \in \R^2$ we have 
\begin{equation}
    (\Lambda u) \, v = \omega v^{\perp}, \qquad \text{where } v^{\perp} = \left( \begin{matrix}
        -v_2
        \\ v_1
    \end{matrix} \right). 
\end{equation}
In the three dimensional case, the antisymmetric part of the gradient is fully characterized by the vorticity $\omega = \curl(u)$, that is 
\begin{equation}
    \Lambda u = \left( \begin{matrix}
        0 & - \omega_3 & \omega_2 
        \\ \omega_3 & 0 & -\omega_1
        \\ - \omega_2 & \omega_1 & 0
        \end{matrix}
        \right), \qquad \text{where } \omega = \left( 
        \begin{matrix} 
        \partial_2 u_3 - \partial_3 u_2
        \\ -\partial_1 u_3 + \partial_3 u_1 
        \\ \partial_1 u_2 - \partial_2 u_1
        \end{matrix}
        \right).  
\end{equation}
Hence, for any vector $v \in \R^3$, we have
\begin{equation}
    (\Lambda u) \, v = \omega \times v. 
\end{equation}

Given an open set $\Omega \subset \R^d$, we denote by $\mathcal{M}(\Omega; \R^N)$ the set of Radon measures (i.e. locally finite Borel measures) on $\Omega$ with values in $\R^N$, endowed with the total variation norm $\norm{\mu}_{\mathcal{M}}$ (see e.g. \cite{EG15}). If the domain and the target space are clear from the context, we adopt the shorthand $\mathcal{M}(\Omega; \R^N) = \mathcal{M}_x$. Given a time interval $I$, we denote by $L^1 (I; \mathcal{M}(\Omega; \R^N)) = L^1_t \mathcal{M}_x$ (whenever $I, \Omega ,N$ are clear from the setting) the collection of weakly measurable curves of Radon measures on $\Omega$ with values in $\R^N$ such that 
$$\int_{I} \norm{\mu_t}_{\mathcal{M}} \, dt < \infty. $$

\subsection{Blow up and precise representative} \label{ss:jumps}

We recall the notion of blow up for $L^1$ functions. 

\begin{definition} \label{d: blow up}
Let $u \in L^1(\Omega; \R^m)$ and $x_0 \in \Omega$. We say that $u$ admits a blow up at $x_0$ if there exists $v \in L^1(B_1)$ such that\footnote{Here, it is crucial to consider $\ell \to 0$, without passing to subsequences.} 
\begin{equation} \label{eq: blow up}
    \lim_{\ell \to 0} \int_{B_1} \abs{u(x_0 + \ell z ) - v(z)} \, dz = 0. 
\end{equation} 
We denote by 
\begin{equation}
    \Omega \setminus \mathcal{N}_u : = \left\{ x_0 \in \Omega \, : \, u \text{ admits a blow up at $x_0$} \right\}.
\end{equation}
\end{definition}

By the above definition, it is clear that whenever $u$ admits a blow up at $x_0$, then the function $v$ is uniquely determined. By the Lebesgue differentiation theorem, the set $\mathcal{N}_u$ is $\mathcal{L}^d$-negligible. More precisely, for $\mathcal{L}^d$-almost every $x_0 \in \Omega \setminus \mathcal{N}_u$, the blow up of $u$ at $x_0$ is a constant function. In order to fix some terminology, we denote by $\S_u$ the complement in $\Omega$ of the set of Lebesgue points for $u$  
\begin{equation}
    \Omega \setminus \mathcal{S}_u : = \left\{ x_0 \in \Omega \colon \text{ $u$ has constant blow up $\bar{u} \in \R^m$ at $x_0$} \right\}.  
\end{equation} 
For any $x_0 \in \Omega\setminus \S_u$, the value $\bar{u}$ is known as the \emph{approximate limit} of $u$ at $x_0$. 

We consider another particularly relevant class of blow up points. We say that $x_0 \in \Omega \setminus \mathcal{N}_u$ is a jump point if there exist $u^+ \neq u^- \in \R^m$ and a direction $\nu \in \mathcal{S}^{d-1} $ such that 
\begin{equation}
    v = u^+ \mathds{1}_{B_1^{\nu_+}} + u^- \mathds{1}_{B_1^{\nu_-}}, \qquad \text{where we set } B_1^{\nu_\pm} =  \{ y \in B_1 \, : \,  \pm \langle y, \nu \rangle \geq 0 \}. 
\end{equation}
We denote by $\J_u$ the collection of jump points of $u$. For any $x_0 \in \mathcal{J}_u$, we say that $u^\pm$ are the \emph{one-sided Lebesgue limits} of $u$ at $x_0$ with respect to $\nu$. Moreover, the triple $(\nu, u^+, u^-)$ is uniquely determined, up to the change of sign of $\nu$ and the permutation of $u^+$ with $u^-$. 

Given $u \in L^1$, it can be proved that the set of points at which $u$ admits a non-constant blow up is $(d-1)$-rectifiable, that is $\J_u$ can be covered by countably many $(d-1)$-dimensional Lipschitz graphs. See \cite{DelNin21} and the references therein for further discussions. We also remark that for a general function $u \in L^1$ without additional properties, we have 
$$\J_u \cup (\Omega \setminus \S_u) \subset \Omega \setminus \mathcal{N}_u, $$
but the inclusion is strict. In the following sections, we recall a classification of blow up points for functions of bounded variation or bounded deformation fields. Sharp estimates of the size of the pathological set $\S_u \setminus \J_u$ will play a key role in our analysis. 

\begin{definition}
Given $u \in L^1(\Omega; \R^m)$, the \emph{precise representative} $\tilde u: \Omega \setminus \mathcal{N}_u \to \R^m$ of $u$ is defined by
\begin{equation} \label{eq: precise representative}
\tilde u(x) : = \fint_{B_1}v(z) \, dz, \qquad \text{where $v$ is the blow up of $u$ at $x$.} 
\end{equation} 
\end{definition}

As discussed above, the precise representative is well-defined $\mathcal{L}^d$-almost everywhere and agrees with $u$ at $\mathcal{L}^d$-almost every point. Moreover, $\tilde{u}$ satisfies 
\begin{equation} \label{eq: explicit formula precise representative}
    \tilde u (x) = \begin{cases}
        \bar{u} & x \in \Omega \setminus \S_u, 
        \\ \displaystyle \frac{u^+ + u^-}{2} & x \in \J_u. 
    \end{cases}
\end{equation}

We introduce the set of admissible convolution kernels
\begin{equation} \label{eq: kernels}
\mathcal{K} := \left\{ \rho \in C^\infty_c(B_1) \, : \, \int_{B_1} \rho(z) \, dz =1,  \rho \text{ even} \right\}. 
\end{equation}
Given $\rho \in \mathcal{K}$ and $\ell>0$, we denote by $\rho_\ell(x) = \ell^{-d} \rho \left( \sfrac{x}{\ell}\right)$. We adopt the standard notation for convolution, that is for any  $u \in L^1_{loc}$ we set $u_\ell = u* \rho_\ell$. The pointwise convergence of the convolutions is related to the existence of blow ups. 
 
\begin{proposition} \label{p: convergence to poinwise repr}
Let $u \in L^1(\R^d)$. For any $\rho \in \mathcal{K}$ and for any $x_0 \in \R^d \setminus \mathcal{N}_u$ it holds 
\begin{equation} \label{eq: u^rho}
    \lim_{\ell \to 0} u * \rho_\ell(x_0) = \int_{B_1} \rho(z) v(z) \, dz = :  u^\rho (x_0), \qquad \text{where $v$ is the blow up of $u$ at $x_0$. } 
\end{equation}
In particular, for any $x_0 \in (\R^d \setminus \S_u) \cup \J_u $ the pointwise limit $u^\rho(x_0)$ is independent of $\rho$ and we have 
\begin{equation}
    \lim_{\ell \to 0} u * \rho_\ell(x_0) = \tilde{u}(x_0). 
\end{equation}
\end{proposition}

\begin{proof}
Fix $x_0 \in \R^d \setminus \mathcal{N}_u$ and denote by $v$ the blow of $u$ at $x_0$. For any $\rho \in \mathcal{K}$, we have 
\begin{align}
    \abs{ u* \rho_\ell (x_0) - \int_{ B_1} \rho(z) v(z) \, dz} & \leq \int_{B_1} \abs{u(x_0 -\ell z) - v(z)} \abs{\rho(z)} \, dz. 
\end{align}
Since $\rho $ is bounded and \eqref{eq: blow up} holds, the latter vanishes in the limit $\ell \to 0$. 
\end{proof}

We will often make use of the following lemma. 

\begin{lemma} \label{L: double mollification} 
Let $u: \R^d \to \R$ be a bounded Borel function and let $\lambda \in \mathcal{M}_x$. Assume that the set $\mathcal{N}_u$ is $\abs{\lambda}$-negligible. Then, for any convolution kernel $\rho \in \mathcal{K}$ with support in $B_{\sfrac{1}{2}}$ we have $u_\ell \, \lambda_\ell \rightharpoonup u^{\rho*\rho} \, \lambda$ weakly as measures, where $u_\ell = u* \rho_\ell, \lambda_\ell = \lambda*\rho_\ell$ and $u^{\rho*\rho}$ are defined by \eqref{eq: u^rho}\footnote{Here, we restrict ourselves to kernels $\rho$ with support in $B_{\sfrac{1}{2}}$ so that $\rho * \rho \in \mathcal{K}$. This choice is made only for convenience.}.
\end{lemma}

\begin{proof}
Fix a kernel $\rho \in \mathcal{K}$ with support in $B_{\sfrac{1}{2}}$. By the properties of convolutions, for any test function $\phi \in C_c(\R^d)$ we have\footnote{Here, it is crucial that $\rho$ is even. \label{f: rho even}} 
\begin{equation}
    \langle \phi, u_\ell \, \lambda_\ell \rangle  = \int_{\R^d} (\phi u_\ell)_\ell \, d \lambda. 
\end{equation}
We claim that $(\phi u_\ell)_{\ell}(x) \to \phi(x) u^{\rho*\rho} (x)$ for $\abs{\lambda}$-almost every $x$. Then, the conclusion follows by the dominated convergence theorem. To check the claim, we split 
\begin{equation}
    (\phi u_\ell )_\ell = \phi \,  (u_\ell)_\ell + \left[ (\phi u_\ell)_\ell - \phi \, (u_\ell)_\ell \right]. 
\end{equation}
We notice that $(u_\ell)_\ell = u*\rho_\ell *\rho_\ell = u* (\rho * \rho)_\ell$. Since $\abs{\lambda}(\mathcal{N}_u) =0$ and $\rho*\rho \in \mathcal{K}$, by \cref{p: convergence to poinwise repr} we infer that $(u_\ell)_\ell(x) \to  u^{\rho * \rho} (x)$ for $\abs{\lambda}$-almost every $x$. Lastly, we have 
\begin{equation} \label{eq: double convolution}
    \norm{((\phi u_\ell)_\ell - \phi \, (u_\ell)_\ell }_{C^0} \leq \norm{u}_{L^\infty} \sup_{x} \int_{B_1} \abs{\phi(x-\ell y) - \phi(x)} \abs{\rho(y)} \, dy
\end{equation}
The latter vanishes in the limit $\ell \to 0$ since $\phi$ is uniformly continuous. 
\end{proof}

\subsection{Functions of bounded variation} \label{ss: BV} 
We list some basic properties of $BV$ functions that will be needed throughout the manuscript. We refer to the monograph \cite{AFP00} for an extensive presentation on this topic. Given an open set $\Omega \subset \R^d$, we say that $u$ has \emph{bounded variation in $\Omega$} ($u \in BV(\Omega)$) if $u \in L^1(\Omega)$ and the distributional gradient $Du$ is a Radon measure on $\Omega$. $BV$ functions are characterized by the fact that the difference quotients are uniformly bounded in $L^1$.  More precisely, given $u \in L^1(\R^d)$ it holds\footnote{A similar characterization holds on general open sets. See also \cite{DRINV23} for more details.  \label{fn: local characterization}} 

\begin{equation} \label{eq: char of BV}
u \in BV(\R^d) \qquad \Longleftrightarrow \qquad \sup_{h \in \R^d \setminus \{0\} } \frac{\norm{u(\cdot + h) - u(\cdot)}_{L^1(\R^d)}}{\abs{h}} < \infty. 
\end{equation}

Functions of bounded variation may experience jumps on sets of codimension one. More precisely, unlike the $W^{1,1}$ case, the jump set $\J_u$ associated to $u \in BV$ may have positive $\mathcal{H}^{d-1}$ measure. We recall the following notation. Given $u \in BV(\Omega; \R^m)$, let $Du = D^a u + D^s u$ be the Radon--Nykodim decomposition of the gradient with respect to the Lebesgue measure, i.e. $D^a u \ll \mathcal{L}^{d}$ and $D^s u \perp \mathcal{L}^d$. The singular part is decomposed as follows 
$$D^s u = D^s u \llcorner (\Omega \setminus \J_u) + D^s u \llcorner \J_u = D^c u + D^j u, $$
where $D^c u$ is the \emph{Cantor part} and $D^j u$ is the \emph{jump part} of the gradient. We also denote by 
$$\tilde D u = D^a u + D^c u$$
the \emph{diffuse part} of the gradient. The following is part of the structure theorem for the gradient of $BV$ functions.

\begin{proposition} [\cite{AFP00}*{Lemma 3.76, Theorem 3.78}] \label{P: prop BV}
Let $u \in BV(\Omega)$. The following holds: 
\begin{enumerate}
    \item the set $\S_u \setminus \J_u$ is $\mathcal{H}^{d-1}$-negligible; 
    \item $\abs{D u} \ll \mathcal{H}^{d-1}$; 
    \item the jump part of the gradient satisfies\footnote{By the definition of $\J_u$, the tensor product $(u^+(x) - u^-(x))\otimes \nu(x)$ is well defined for $\mathcal{H}^{d-1}$-a.e. $x \in \mathcal{J}_u$. } 
    \begin{equation}
        D^j u = (u^+ - u^-) \otimes \nu \,    \mathcal{H}^{d-1} \llcorner \J_u. 
    \end{equation}
\end{enumerate}
\end{proposition}

\begin{remark} \label{R: N is Du negligible}
By the above proposition, we have $\abs{ D u } (\mathcal{N}_{u} ) =0 $ and $\abs{D u}$-almost every $x$ is either a Lebesgue point or a jump point for $u$. Hence, the precise representative $\tilde u$  satisfies \eqref{eq: explicit formula precise representative} $\abs{D u}$-almost everywhere. 
\end{remark}

We also recall the chain rule for $BV$ functions. We point out that the diffuse part satisfies the same chain rule as in the smooth setting, whereas the chain rule for the jump part is different. This fact will play an important role in the following sections. 

\begin{proposition} \cite{AFP00}*{Theorem 3.96} \label{p: BV chain rule}
Let $u \in L^\infty \cap BV(\Omega; \R^m)$ and let $f: \R^m \to \R^n$ be a $C^1$ function. Then $ f \circ u \in BV$ and the following chain rule holds 
\begin{equation}
    \tilde{D} (f \circ u) = \nabla f (\tilde u) \, \tilde D u, \qquad D^j (f \circ u) = (f(u^+) - f(u^-)) \otimes \nu \, \mathcal{H}^{d-1} \llcorner \J_u. 
\end{equation}
\end{proposition}

\subsection{Vector fields with bounded deformation} \label{ss: BD}

We recall some fundamental properties of $BD$ vector fields. We refer to \cites{ACDalM97, TS80} for a detailed exposition. The basic theory is rather similar to that of $BV$ functions, with few exceptions. Given an open set $\Omega \subset \R^d$, we say that $u$ has \emph{bounded deformation in $\Omega$} if $u \in L^1(\Omega)$ and the symmetric distributional gradient $Eu$ is a Radon measure on $\Omega$. $BD$ vector fields enjoy the following characterization in terms of \emph{longitudinal} difference quotients. Given $u \in L^1(\R^d)$, we have\footref{fn: local characterization}  
\begin{equation} \label{eq: char of BD}
    u \in BD(\R^d) \qquad \Longleftrightarrow \qquad \sup_{h \in \R^d \setminus \{0\}} \frac{\norm{\langle h , u(\cdot + h) - u(\cdot)\rangle}_{L^1(\R^d)}}{\abs{h}^2} < \infty.  
\end{equation}
See e.g. \cite{DRInvN24}*{Lemma 2.11} for further details. The following is part of the structure theorem for the symmetric gradient of $BD$ fields.   

\begin{proposition} \cite{ACDalM97}*{Theorem 6.1, Equation (4.2)} \label{P: prop BD}
Let $u \in BD(\Omega; \R^d)$. The following holds: 
\begin{enumerate}
    \item the set $\S_u \setminus \J_u$ is $\abs{E u}$-negligible\footnote{To the best of the author's knowledge, it is currently not known whether $\S_u \setminus \J_u$ is $\mathcal{H}^{d-1}$-negligible, as in the $BV$ setting.}; 
    \item the restriction of the symmetric gradient to $\J_u$ satisfies
    \begin{equation} \label{eq: symmetric gradient jump} 
        E u \llcorner \J_u = \frac{(u^+ - u^-) \otimes \nu + \nu \otimes (u^+ - u^-)}{2} \, \mathcal{H}^{d-1} \llcorner \J_u; 
    \end{equation}
    \item if $\div(u) = 0 $, then it holds 
    \begin{equation} \label{eq: div-free traces}
        u^+ \cdot \nu = u^- \cdot \nu \qquad \text{for $\mathcal{H}^{d-1}$ almost every $x \in \J_u$. }  
    \end{equation} 
\end{enumerate} 
\end{proposition}

\begin{remark}
The first statement of the above proposition is analogous to \cref{R: N is Du negligible} for $BV$ functions. Hence, $\abs{E u}$-almost every $x \in \Omega$ is either a Lebesgue point or a jump point for $u$ and the precise representative $\tilde u$ satisfies \eqref{eq: explicit formula precise representative} $\abs{E u}$-almost everywhere. The second statement concerns the characterization of the symmetric gradient on Lipschitz hypersurfaces of codimension one. The last statement follows by computing the trace of the symmetric tensor product defining the density in \eqref{eq: symmetric gradient jump}\footnote{This fact is a particular instance of a general property of measure divergence vector fields with strong traces on hypersurfaces. $BD$ vector fields admit bilateral traces on any Lipschitz hypersurface $\Sigma$ and the normal component of these traces is continuous across $\Sigma$ if and only if $\div(u)$ does not give mass to $\Sigma$. See e.g. \cite{DRInvN24} and the references therein.}.
\end{remark}

\section{Proofs} \label{s: proofs}

\subsection{A direct proof via the chain rule} \label{s: direct proof}

In this section, we discuss a direct proof of \cref{t:main intro 1} based only on the chain rule for $BV$ functions. 

\begin{lemma} \label{L: chain rule divergence}
Let $u \in L^\infty \cap BV(\Omega)$. With the notation of \cref{ss:jumps} and \cref{ss: BV}, the following identities hold in the sense of measures 
\begin{align} 
    \div(u \otimes u) & = (\tilde u \cdot D) \, u + \tilde{u} \div(u), \label{eq: chain rule 1}
    \\ D \frac{\abs{u}^2}{2} & = \tilde{u}^T  D u, \label{eq: chain rule 2}
    \\ \div\left( u \frac{\abs{u}^2}{2} \right) & = \tilde{u}\cdot D \frac{\abs{u}^2}{2} + \frac{\abs{\tilde u}^2}{2} \div(u) + \frac{1}{8} \abs{u^+-u^-}^2 \div^j (u). \label{eq: chain rule 3} 
\end{align}
\end{lemma}

\begin{proof}
We decompose the derivatives on the left hand side in the diffuse and jump parts, as described in \cref{ss: BV}. In virtue of \cref{p: BV chain rule}, the diffuse parts follow the chain rule for smooth functions, provided that we consider the precise representative $\tilde{u}$. Then, we only have to check the behaviour of the jump part. To begin, we study \eqref{eq: chain rule 1}. We denote by $u^\pm_n = u^{\pm}\cdot \nu$. Then, by \cref{P: prop BV} and \cref{p: BV chain rule} we have 
\begin{align}
    \div^j (u \otimes u) & = (u^+ u^+_n - u^- u^-_n) \, \mathcal{H}^{d-1}\llcorner \J_u 
    \\ & = \left[ ((u^+- u^-) \otimes \nu) \tilde{u} +  (u^+_n - u^-_n)\tilde{u}  \right] \mathcal{H}^{d-1}\llcorner \J_u
    \\ & = (\tilde{u} \cdot D^j)\, u + \tilde{u} \, \div^j(u),  
\end{align}
thus showing \eqref{eq: chain rule 1} at the level of the jump parts. Regarding \eqref{eq: chain rule 2}, we have 
\begin{align}
    D^j \frac{\abs{u}^2}{2} & = \left(\frac{\abs{u^+}^2}{2} - \frac{\abs{u^-}^2}{2}\right) \nu \, \mathcal{H}^{d-1}\llcorner \J_u = \tilde{u}^T [ (  u^+-u^-)\otimes \nu] \, \mathcal{H}^{d-1}\llcorner \J_u = \tilde{u}^T D^j u,   
\end{align}
thus proving \eqref{eq: chain rule 1} for the jump parts. Regarding \eqref{eq: chain rule 3}, the jump part of the left hand side reads as 
\begin{align}
    \div^j\left( u \abs{u}^2 \right) = \left( \abs{u^+}^2 u^+_n - \abs{u^-}^2 u^-_n \right) \mathcal{H}^{d-1}\llcorner \J_u.   
\end{align}
On the other hand, using \eqref{eq: chain rule 2} we have 
\begin{equation}
    \tilde{u} \cdot D^j \abs{u}^2 + \abs{\tilde u}^2 \div^j (u) = \left[ 2 \tilde{u}\otimes \tilde u : (u^+ - u^-) \otimes \nu + \abs{\tilde{u}}^2 (u^+_n - u^-_n) \right] \mathcal{H}^{d-1}\llcorner \J_u.  
\end{equation}
By a direct computation, it is readily checked that  
\begin{equation}
    \left( \abs{u^+}^2 u^+_n - \abs{u^-}^2 u^-_n \right) - \left[ 2 \tilde{u}\otimes \tilde u : (u^+ - u^-) \otimes \nu + \abs{\tilde{u}}^2 (u^+_n - u^-_n) \right] = \frac{\abs{u^+ - u^-}^2}{4} (u^+_n - u^-_n). 
\end{equation}
Then we also have \eqref{eq: chain rule 3} at the level of the jump parts. 
\end{proof}

We give a direct proof of \cref{t:main intro 1} and postpone further comments to \cref{R: on the failure of the chain rule}. 

\begin{proof}[Proof of \cref{t:main intro 1}] 
Using \eqref{eq: chain rule 1} for almost every time slice, we write the momentum equation as 
\begin{equation}
    \partial_t u + (\tilde{u} \cdot D) \, u + \nabla p = f.
\end{equation} 
Then, mollifying and multiplying by $u_\ell$, we find\footnote{By mollification estimates, it holds $\partial_t u_\ell \in L^\infty_{x,t}$. Hence, $u_\ell$ is Lipschitz continuous in space-time and $ 2 u_\ell \partial_t u_\ell = \partial_t \abs{u_\ell}^2$. \label{f: u_e is regular in time}} 
\begin{equation}
    \partial_t \frac{\abs{u_\ell}^2}{2} + ((\tilde{u} \cdot D) \, u)_\ell \cdot u_\ell + \div(u_\ell p_\ell) = f_\ell \cdot u_\ell. 
\end{equation}
By the properties of convolutions, we have 
$$\partial_t \frac{\abs{u_\ell}^2}{2} + \div (u_\ell p_\ell) - f_\ell \cdot u_\ell \rightharpoonup \partial_t \frac{\abs{u}^2}{2} + \div(up) - f\cdot u  \qquad \text{in } \mathcal{D}'_{x,t}. $$
Then, we manipulate the cubic term. We fix a generic time slice. By \cref{P: prop BV}, we have $\abs{D u_t}(\mathcal{N}_{u_t}) = 0$ and the measure $(\tilde u_t \cdot D) u_t$ is absolutely continuous with respect to $\abs{D u_t}$. Hence, by \cref{L: double mollification}, we deduce 
\begin{equation} \label{eq: cubic convergence}
    ((\tilde u_t \cdot D) \, u_t)_\ell \cdot (u_t)_\ell \rightharpoonup ((\tilde{u}_t \cdot D) \, u_t )\cdot \tilde{u}_t
\end{equation} 
weakly as measures in $\Omega$. Then, by \eqref{eq: chain rule 2} and \eqref{eq: chain rule 3}, for almost every time slice we write
\begin{equation}
    ((\tilde{u}_t \cdot D)\, u_t )\cdot \tilde{u}_t = \tilde{u}_t \otimes \tilde{u}_t : Du_t = \tilde{u}_t \cdot D \frac{\abs{u_t}^2}{2} = \div\left( u_t \frac{\abs{u_t}^2}{2} \right). 
\end{equation}
To conclude, by uniform $L^\infty$ bounds, we apply the dominated convergence theorem with respect to the time variable and we infer that 
\begin{equation}
    ((\tilde u\cdot D) \, u)_\ell \cdot u_\ell \rightharpoonup \div\left( u \frac{\abs{u}^2}{2}\right) \qquad \text{in } \mathcal{D}'_{x,t}. 
\end{equation}
This concludes the proof. 
\end{proof}

\begin{remark} \label{R: on the failure of the chain rule}
By \cref{L: chain rule divergence}, it is clear that the classical chain rule for $\div(u \otimes u)$ always holds in the $BV$ setting, but the chain rule for $\div(u \abs{u}^2)$ fails whenever $\div(u)$ has a non-trivial jump part. However, the formulas in \cref{L: chain rule divergence} also hold in the one-dimensional case, showing that 
\begin{equation}
    \tilde u \, \partial_x \frac{\abs{u}^2}{2} = \partial_x \frac{u^3}{3} - \frac{1}{12} \left[ u^+ - u^- \right]^3 \nu \, \mathcal{H}^0 \llcorner \mathcal{J}_u. 
\end{equation}
Then, working with space-time measures instead of fixing time slices, the chain rules established in \cref{L: chain rule divergence} can be used to show that for $ u \in L^\infty_{x,t} \cap BV_{x,t}$\footnote{This is equivalent to requiring that $u \in L^{\infty}_{x,t} \cap L^1_t BV_x$. } solving $1d$ Burgers it holds 
\begin{equation}
    \partial_t \frac{u^2}{2} + \partial_x \frac{u^3}{3} = \frac{1}{12} \left[ u^+ - u^- \right]^3 \nu_x \, \mathcal{H}^1 \llcorner \mathcal{J}_u,
\end{equation}
where $\nu = (\nu_t, \nu_x) \in \mathbb{S}^1$ is the normal vector to the jump set in space-time. This is consistent with shocks having non-trivial energy dissipation concentrated on the jump set (see e.g. \cite{Daf16}) and with the fact that continuous Burgers solutions do not dissipate energy \cite{Daf06}. 
\end{remark}

\subsection{Velocity fields with bounded deformation} \label{s: velocity BD}

This section aims to give a detailed and self-contained proof of \cref{thm: main BD}. Our analysis relies on the following integral computations.

\begin{lemma} \label{l: averaged convergence}
Let $u\in L^\infty(\R^d; \R^d)$ and $ x_0 \in \R^d \setminus (\S_u \setminus \J_u)$. Denote by $(\nu, u^+, u^-) \in \mathbb{S}^{d-1} \times \R^d \times \R^d$ the one-sided limits of $u$ with respect to $\nu$, with the convention that $u^+ = u^-$ and $\nu = e_d$ if $x_0 \in \R^d \setminus \S_u$. Fix a radial convolution kernel $\rho \in \mathcal{K}$. Then, for any $y \in B_1$ we have
\begin{equation} \label{eq: shifted blow up}
    \lim_{\ell \to 0} \int_{B_1} u( x_0 + \ell (y-z) ) \rho(z)\, dz = \left( \int_{B_1(y) \cap \R^d_+ } \rho(y-z) \, dz \right) u^+ + \left( \int_{B_1(y) \cap \R^d_- } \rho(y - z) \, dz \right) u^-,  
\end{equation}
where we set $\R^d_{\pm} = \{ \pm z \cdot e_d \geq 0 \}$. 
\end{lemma}

\begin{proof}
Without loss of generality, we may assume $x_0 =0$ and $\nu = e_d$\footnote{Since $\rho$ is radial, also the right hand side in \eqref{eq: shifted blow up} invariant under rotation.}. Then, it holds 
\begin{equation}
    \int_{B_1} u(\ell (y-z)) \rho(z) \, dz = \left( \int_{B_1(y) \cap \R^d_+} + \int_{B_1(y) \cap \R^d_-}\right) u(\ell z) \rho(y-z)\, dz.  
\end{equation}
We compute both terms separately. We have 
\begin{align}
    \abs{ \int_{B_1(y) \cap \R^d_+} u(\ell z) \rho(y-z) \, dz - \left( \int_{B_1(y) \cap \R^d_+ } \rho(y-z) \, dz \right) u^+  } & \leq \int_{B_1(y)\cap \R^d_+} \abs{u(\ell z) - u^+} \abs{\rho(y-z)} \, dz 
    \\ & \leq \norm{\rho}_{L^\infty} \int_{B_2^+} \abs{u(\ell z) - u^+}\, dz. 
\end{align}
The latter vanishes in the limit $\ell \to 0$. Similarly, we show that 
$$\lim_{\ell \to 0} \int_{B_1(y) \cap \R^d_-} u(\ell z) \rho(y-z) \, dz = \left( \int_{B_1(y) \cap \R^d_- } \rho(y-z) \, dz \right) u^-. $$
This concludes the proof. \end{proof}

\begin{proposition} \label{P: limit reynolds}
Let $\rho \in \mathcal{K}$ be a radial kernel with support in $B_{\sfrac{1}{2}}$. Under the assumptions of \cref{l: averaged convergence}, it holds 
\begin{equation}
\lim_{\ell \to 0} [(u \otimes u)_{\ell} -  u_\ell \otimes u_\ell ]_\ell (x_0) = \bar{\alpha} \, (u^+ - u^-) \otimes (u^+ - u^-),
\end{equation}
where we set 
\begin{equation} \label{eq: alpha constant}
    \alpha(y) : = \int_{B_1(y) \cap \R^d_+} \rho(y-z) \, dz, \qquad \bar{\alpha} : = \int_{B_1} \alpha(y) (1- \alpha(y)) \rho(y) \, dy. 
\end{equation}
\end{proposition}

\begin{proof}
Without loss of generality, we may assume $x_0=0$ and $\nu = e_d$. We study both terms separately. 

\textsc{\underline{The first term}:} Since $x_0 \in \R^d \setminus (\S_u \setminus \J_u)$ and $u \in L^\infty$, it is easy to check that $x_0$ is a Lebesgue point or a jump discontinuity for $u \otimes u$. More precisely, $u^+ \otimes u^+, u^-\otimes u^-$ are the one-sided limits of $u \otimes u$ with respect to $\nu$, with the convention that $u^+ \otimes u^+ = u^- \otimes u^- = \bar{u} \otimes \bar{u}$ whenever $x_0 \in \R^d \setminus \S_u$. In other words, with our notation, the precise representative of $u \otimes u$ satisfies 
\begin{equation}
    \widetilde{u \otimes u} = \frac{u^+ \otimes u^+ + u^- \otimes u^-}{2} \qquad \text{in } \R^d \setminus (\S_u \setminus \J_u). 
\end{equation}
By the properties of convolutions, we write 
\begin{equation}
    ((u \otimes u)_\ell )_{\ell} = (u \otimes u) * \eta_\ell \qquad \text{where } \eta = \rho *\rho.  
\end{equation}
Since $\eta$ is a radial kernel in $ \mathcal{K}$, by \cref{p: convergence to poinwise repr}, it follows 
\begin{equation}
    \lim_{\ell \to 0} [(u \otimes u) * \eta_\ell] (0) = \frac{u^+ \otimes u^+ + u^- \otimes u^-}{2}.  
\end{equation}

\textsc{\underline{The second term}:}
Since $\rho$ is even, we have 
\begin{align}
    [u_\ell \otimes u_\ell]_\ell(0) & = \int_{B_1} u_\ell(\ell y) \otimes u_{\ell}(\ell y) \rho(y) \, dy 
    \\ & = \int_{B_1} \left( \int_{B_1} u(\ell(y-z)) \rho(z) \, dz \right) \otimes \left( \int_{B_1} u(\ell (y-z)) \rho(z) \, dz  \right) \rho(y)\, dy.  
\end{align}
By \cref{l: averaged convergence} and the dominated convergence theorem (recall that $u \in L^\infty$), we infer that 
\begin{align}
    \lim_{\ell \to 0} [u_\ell \otimes u_\ell]_\ell (0) & = \int_{B_1} \left[ \alpha(y) u^+ + (1-\alpha(y)) u^- \right] \otimes \left[ \alpha(y) u^+ + (1-\alpha(y)) u^- \right] \rho(y) \, dy 
    \\ & = \int_{B_1} \left[\alpha^2 u^+ \otimes u^+ + (1-\alpha)^2 u^-\otimes u^- + \alpha(1-\alpha) (u^+\otimes u^- + u^-\otimes u^+) \right] \rho(y) \, dy. 
\end{align}
We claim that for any $y \in B_1$ we have
$$\alpha(y) = 1 - \alpha(-y). $$
Indeed, we have 
\begin{align}
    \alpha(y) + \alpha(-y) & = \int_{B_1(y)\cap \R^d_+} \rho(y-z) \, dz + \int_{B_1(-y) \cap \R^d_+} \rho(-y-z) \, dz 
    \\ & = \int_{B_1(y)\cap \R^d_+} \rho(y-z) \, dz + \int_{B_1(y) \cap \R^d_-} \rho(-y+z) \, dz 
    \\ & = \int_{B_1(y) \cap \R^d_+} \rho(y-z) \, dz + \int_{B_1(y) \cap \R^d_-} \rho(y-z) \, dz = 1. 
\end{align}
Therefore, splitting the integral in the contributions over $B_1^+$ and $B_1^-$ and using symmetries, we obtain 
\begin{align}
    \lim_{\ell \to 0} & (u_\ell \otimes u_\ell)_\ell (0)  = \int_{B_1^+}\left[  (\alpha^2 + (1-\alpha)^2 ) \left( u^+ \otimes u^+ + u^- \otimes u^- \right) + 2 \alpha (1-\alpha) \left( u^+ \otimes u^- + u^-\otimes u^+ \right) \right]  \rho(y) \, dy 
    \\ & = \frac{u^+ \otimes u^+ + u^-\otimes u^-}{2} + \left( \int_{B_1^+} 2 \alpha(1-\alpha) \rho \, dy \right) \left( u^+ \otimes u^- + u^- \otimes u^+ - u^+ \otimes u^+ - u^- \otimes u^- \right)
    \\ & = \frac{u^+ \otimes u^+ + u^-\otimes u^-}{2} - \left( \int_{B_1}  \alpha(1-\alpha) \rho \, dy \right) \left[ (u^+ - u^-)\otimes (u^+ - u^-) \right].  
\end{align}
This concludes the proof. 
\end{proof}

We discuss the proof of \cref{thm: main BD}. 

\begin{proof} [Proof of \cref{thm: main BD}] 
Fix a radial convolution kernel $\rho \in \mathcal{K}$ and a test function $\phi \in C^\infty_c(\Omega \times (0,T))$. Since $-R^\ell  = (u\otimes u)_\ell - u_\ell \otimes u_\ell$ is  symmetric matrix, by the Constantin--E--Titi approximation \eqref{eq: CET approx} and the properties of convolutions, we have 
\begin{align}
    \langle \D , \phi \rangle & = \lim_{\ell \to 0} \int_0^T \int_\Omega \phi R^\ell : (E u)_\ell \, dx \, dt
    \\ & = \lim_{\ell \to 0} \int_0^T \int_\Omega \left[ (\phi R^\ell)_\ell - \phi \, (R^\ell)_\ell \right] : \, d (E u_t) \, dt + \lim_{\ell \to 0} \int_0^T \int_\Omega \phi \, ( R^\ell )_\ell : \, d (E u_t) \, dt. 
\end{align} 
We study separately both terms. As shown in \eqref{eq: double convolution}, it can be checked that for any time slice $(\phi R^\ell)_\ell - \phi \, (R^\ell)_\ell$ converges uniformly to $0$. Since $\phi, R^\ell$ are uniformly bounded in $\Omega \times (0,T)$, by the dominated convergence theorem, we conclude that 
$$\lim_{\ell \to 0} \int_0^T \int_\Omega \left[ (\phi R^\ell)_\ell - \phi \, (R^\ell)_\ell \right] : \, d (E u_t) \, dt =0. $$
The second term requires more careful. By dominated convergence in time, it is enough to check that 
\begin{equation} \label{eq: reduction to time slice}
    \lim_{\ell \to 0} \int_\Omega \phi \, (R^\ell)_\ell : \, d(Eu_t) = 0 \qquad \text{ for $\mathcal{L}^1$-almost every $t \in (0,T)$.} 
\end{equation}
Hence, we fix a time slice such that $u_t \in L^\infty_x \cap BD_x$ and $\div(u_t) =0$. By \cref{P: limit reynolds}, we infer 
\begin{equation}
    \lim_{\ell \to 0} ((u\otimes u)_\ell - u_\ell \otimes u_\ell)_\ell  = \begin{cases}
        0 & x \in \Omega \setminus \mathcal{S}_{u_t}, 
        \\  \bar{\alpha} \, (u^+ - u^-)\otimes (u^+ - u^-) & x \in \mathcal{J}_{u_t},   
    \end{cases}
\end{equation}
where $\bar{\alpha}$ is the constant given by \eqref{eq: alpha constant}. Since $\abs{E u_t}(\mathcal{S}_{u_t} \setminus \mathcal{J}_{u_t}) = 0$ by \cref{P: prop BD}, by the dominated converge theorem and the characterization of the symmetric gradient on the jump set \eqref{eq: symmetric gradient jump}, we obtain 
\begin{align}
    - \lim_{\ell \to 0} \int_{\Omega} \phi \,(R^\ell)_\ell \, d (E u_t) & = \bar{\alpha} \int_{\J_{u_t}} \phi (u^+ - u^-) \otimes (u^+- u^-) : \, d ( E u_t) 
    \\ & = \bar{\alpha} \int_{\J_{u_t}}  \phi \, (u^+ - u^-)\otimes (u^+ - u^-) : \left[ \frac{(u^+ - u^-)\otimes \nu + \nu \otimes (u^+ - u^-)}{2} \right] \, d \mathcal{H}^{d-1} 
    \\ & = \bar{\alpha} \int_{\J_{u_t}} \phi \, \abs{u^+ - u^-}^2 \langle u^+ - u^- , \nu \rangle \, d \mathcal{H}^{d-1} =0,   
\end{align}
since $u_t$ is divergence-free and \eqref{eq: div-free traces} holds. This concludes the proof. 
\end{proof}

\subsection{Measure antisymmetric gradient} \label{s: vortex sheets}

In this section, we study the dissipation associated to bounded Euler solutions with measure antisymmetric gradient. We prove the following result. 

\begin{theorem} \label{t:main general}
Let $u \in L^\infty_{x,t}$ be a weak solution to \eqref{E} with $p,f \in L^1_{x,t}$. Assume that the spatial antisymmetric gradient $\Lambda u$ is in $L^1_t \mathcal{M}_x$ and the set $\mathcal{N}_{u_t}$ is $\abs{\Lambda u_t}$-negligible for almost every $t$. Then $\D \equiv 0$. 
\end{theorem}

\begin{remark} \label{R: particular cases}
\cref{t:main general} applies to the following cases. 
\begin{enumerate}
    \item[i)] If $u \in L^\infty_t C^0_x$ with $\Lambda u \in L^1_t \mathcal{M}_x$, then the set $\mathcal{N}_{u_t}$ is empty. 
    \item[ii)] If $u \in L^\infty_{x,t} \cap L^1_t BV_x$, then the set $\mathcal{N}_{u_t}$ is $\abs{D u_t}$-negligible by \cref{R: N is Du negligible}, thus giving another proof of \cref{t:main intro 1}. 
    \item[iii)] If $u \in L^\infty_{x,t}$ with $\Lambda u \in L^1_{x,t}$, then the conclusion is trivial, since the set $\mathcal{N}_{u_t}$ is always $\mathcal{L}^d$-negligible. 
\end{enumerate}
\end{remark}

As mentioned in \cref{ss: generalized vortex sheets}, the proof of \cref{t:main general} is a consequence of the following manipulations of the momentum equation.

\begin{proposition} \label{p: rewriting the equation}
Fix a kernel $\rho \in \mathcal{K}$ with support in $B_{\sfrac{1}{2}}$. Let $\eta = \rho *\rho \in \mathcal{K}$ and let $u^\eta$ be defined by \eqref{eq: u^rho}. Under the assumptions of \cref{t:main general}, it holds 
\begin{equation} \label{eq: modified euler bis}
    \partial_t u + \nabla P + (\Lambda u)\, u^{\eta} = f \qquad \text{in} \quad \mathcal{D}'_{x,t}.   
\end{equation} 
\end{proposition}

\begin{proof}
Mollifying the momentum equation with respect to the spatial variable, we have 
\begin{equation}
    f_\ell = \partial_t u_\ell + \div(u_\ell \otimes u_\ell) + \nabla p_\ell - \div(R^\ell), \qquad \text{where } R^\ell = u_\ell \otimes u_\ell - (u\otimes u)_\ell. 
\end{equation}
Since $u_\ell$ is smooth in space and divergence-free, a direct computation shows that\footnote{To be precise, the divergence-free condition is only needed to write $ \div(u_\ell \otimes u_\ell) = (u_\ell \cdot \nabla) \, u_\ell$.} 
\begin{equation}
    \div(u_\ell \otimes u_\ell) = (u_\ell \cdot \nabla) \, u_\ell = \nabla \left( \frac{\abs{u_\ell}^2}{2}\right) + (\Lambda u_\ell) \, u_\ell. 
\end{equation}
Hence, we have 
\begin{equation}
    f_\ell = \partial_t u_\ell + \nabla \left( \frac{\abs{u_\ell}^2}{2} + p_\ell \right) + (\Lambda u_\ell) \, u_\ell - \div(R^\ell). 
\end{equation}
Then, letting $\ell \to 0$, we have
\begin{equation}
    f = \partial_t u + \nabla P + \lim_{\ell \to 0} (\Lambda u_\ell) \, u_\ell, 
\end{equation}
where the limits are taken in the sense of distributions. To conclude, we show that 
\begin{equation} \label{eq: space time distributional convergence}
    \lim_{\ell \to 0} (\Lambda u_\ell) \, u_\ell = (\Lambda u) \, u^\eta \qquad \text{in } \mathcal{D}'_{x,t}. 
\end{equation} 
Using \cref{L: double mollification} with $\lambda = \Lambda u_t$ and since $\abs{\Lambda u_t} (\mathcal{N}_{u_t}) =0 $ by assumption for almost every $t$, we have 
\begin{equation}
    (\Lambda u_t)_\ell \,  (u_t)_\ell \rightharpoonup (\Lambda u_t) \,  u_t^\eta \qquad \text{for almost every $t$.}
\end{equation}
Then, by uniform $L^\infty$ bounds, we apply the dominated convergence theorem with respect to the time variable and  \eqref{eq: space time distributional convergence} follows. 
\end{proof}

In conclusion, we discuss the proof of \cref{t:main general}.  

\begin{proof} [Proof of \cref{t:main general}]
We fix a convolution kernel $\rho \in \mathcal{K}$ with support in $B_{\sfrac{1}{2}}$. We let $\eta = \rho * \rho$ and we derive \eqref{eq: modified euler bis} accordingly. Then, we mollify the momentum equation with $\rho_\ell$ in the form of \eqref{eq: modified euler bis} and we multiply by $u_\ell = u*\rho_\ell$. Hence, since $u_\ell$ is divergence-free, we have\footref{f: u_e is regular in time}
\begin{equation}
    f_\ell \cdot u_\ell = u_\ell \cdot \partial_t u_\ell + u_\ell \cdot \nabla P_\ell + ((\Lambda u) \, u^\eta)_\ell \cdot u_\ell = \partial_{t} \frac{\abs{u_\ell}^2}{2} + \div(u_\ell P_\ell) + ((\Lambda u) \, u^\eta)_\ell \cdot u_\ell. 
\end{equation} 
Then, by the properties of convolutions, letting $\ell \to 0$ we find 
\begin{equation}
f \cdot u = \partial_t \frac{\abs{u}^2}{2} + \div(u P) + \lim_{\ell \to 0} ( (\Lambda u) \,  u^\eta )_\ell \cdot u_\ell,
\end{equation}
where the limit is taken in $\mathcal{D}'_{x,t}$. Using again \cref{L: double mollification} as in the proof of \cref{p: rewriting the equation}, for almost every $t$ we have 
$$( (\Lambda u_t) \, u^\eta_t)_\ell \cdot (u_t)_\ell \rightharpoonup ( (\Lambda u_t ) \, u^\eta_t ) \cdot u^\eta_t = \Lambda u_t : (u^\eta_t \otimes u^\eta_t) =0, $$
since $\Lambda u_t$ is an antisymmetric matrix-valued measure and $u^\eta_t \otimes u^\eta_t$ is a symmetric matrix-valued Borel function, which is well-defined $\abs{\Lambda u_t}$-almost everywhere by assumption. Then, using again dominated convergence in time, we conclude that $((\Lambda u) \, u^\eta)_\ell \cdot u_\ell \rightharpoonup ((\Lambda u) \, u^\eta) \cdot u^\eta  = 0 $ in $\mathcal{D}'_{x,t}$. 
\end{proof}

\begin{remark}
We point out that the proofs of \cref{p: rewriting the equation} and \cref{t:main general} can be easily adapted whenever $\omega \abs{u}^2 \in L^1_{x,t}$. For instance, in the two-dimensional case, $\omega \in L^\infty_t L^{\sfrac{3}{2}}_x$ implies $u \in L^\infty_t L^6_x$ by Calder\'on--Zygmund estimates and Sobolev embedding, thus $\omega \abs{u}^2 \in L^\infty_t L^1_{x}$. However, it is well-known that $\omega \in L^{\sfrac{3}{2}}_x$ gives $u \in W^{\sfrac{1}{3}, 3}_x$. Hence, $\D$ vanishes by the classical results \cites{CET94, DR00}. The argument outlined above provides an alternative proof of this fact. We leave the details to the interested reader.  
\end{remark}

\bibliographystyle{plain} 
\bibliography{biblio}

\end{document}